\documentclass[12pt]{amsart}
\usepackage{amsmath,amssymb,latexsym}
\usepackage{fullpage}

\theoremstyle{plain}
\newtheorem{theorem}{Theorem}
\newtheorem{proposition}{Proposition}
\newtheorem{lemma}{Lemma}

\def\A{\mathbb{A}}
\def\C{\mathbb{C}}

\def\Mat{\text{Mat}}

\begin{document}

\title{On the genericity of Eisenstein series and their residues for covers of $GL_m$}
\author{Solomon Friedberg}
\author{David Ginzburg}
\address{Department of Mathematics, Boston College, Chestnut Hill, MA 02467-3806, USA}
\address{School of Mathematical Sciences, Tel Aviv University, Ramat Aviv, Tel Aviv 6997801,
Israel}
\thanks{This work was supported by the US-Israel Binational Science Foundation,
grant number 2012019, and by the National Security Agency, grant number
H98230-13-1-0246 and the National Science Foundation, grant number 1500977 (Friedberg).}
\begin{abstract}
Let $\tau_1^{(r)}$, $\tau_2^{(r)}$ be two genuine cuspidal automorphic representations on $r$-fold covers of the adelic points of
the general linear groups $GL_{n_1}$, $GL_{n_2}$, resp., 
and let $E(g,s)$ be the associated Eisenstein series on an $r$-fold cover of $GL_{n_1+n_2}$.  Then the value or
residue at any point $s=s_0$ of $E(g,s)$ is an automorphic form, and generates an automorphic representation.
In this note we show that if $n_1\neq n_2$ these automorphic representations (when not identically zero) are generic, while
if $n_1=n_2:=n$ they are generic except for residues at $s=\frac{n\pm1}{2n}$.
\end{abstract}

\maketitle

\section{introduction}\label{intro}

The study of the
Whittaker coefficients of Eisenstein series has been a major tool in the theory of automorphic forms.
See Shahidi \cite{Sh} and the references there.  Let $G$ be a reductive group defined over a number field $F$, 
$P$ be a maximal parabolic subgroup of $G$ with Levi decomposition $P=MU$, and let $\tau$
be an automorphic representation of  $M(\A)$.  Given $f_\tau$ in the space of $\tau$, one may form an 
Eisenstein series $E(g,s,f_\tau)$ on $G(\A)$; this is an automorphic form defined by an absolutely convergent 
sum for $\Re(s)$ sufficiently large and by analytic continuation in general (Langlands \cite{Lan}).  If $\tau$ is generic, then the global 
Whittaker coefficient of $E$ is computed by unfolding
the Eisenstein series and making use of the factorization of a global Whittaker functional into local ones,
relating it to integrals involving local Whittaker functionals for $\tau$.  These are computed at
unramified places using the Casselman-Shalika formula. One
finds that the Whittaker coefficient of $E$ may be expressed in terms of
certain Langlands $L$-functions for $\tau$.   Thus the question of whether or not the Eisenstein 
series (resp.\ its residue) at a given point is generic is related to the properties of the $L$-functions which appear.

The goal of this article is to study when a maximal parabolic Eisenstein series or its residue is generic for metaplectic covers
of the general linear group.  Recall that given a positive integer $r$, a cover of $GL_n(\A)$ may be defined whenever
$F$ contains a full set of $r$-th roots of unity.  However, many aspects of the
theory of automorphic forms change for covers.  In particular, the local
Whittaker model is not unique (this happens since the inverse image of the full torus in the cover is no longer abelian), 
and the approach outlined above does not apply. The global Whittaker coefficients may be computed as a Dirichlet series
for $\Re(s)$ sufficiently large (Brubaker-Friedberg \cite{B-F}), in fact one involving exponential sums of arithmetic interest, but it is not apparent how one could use this expression to determine non-vanishing at a particular point.
Also, though there
is a good understanding of the constant terms of the metaplectic Eisenstein series (Gao \cite{Gao}) the residual spectrum has been
computed in only a few cases (see for example Gao \cite{Gao2}), and in general it is not clear how to find
the Whittaker coefficients of the residues.  Indeed, the determination of such coefficients has been a long-standing open problem
even for low degree covers of $GL_2$ (see Eckhardt-Patterson \cite{E-P}, Chinta-Friedberg-Hoffstein \cite{C-F-H}).

In this paper we offer a new approach to the question of genericity.  This is based on the classification 
of Fourier coefficients via unipotent orbits. In our main result, we give a complete description of when such an Eisenstein series 
or its residue
may be generic.  Theorem~\ref{th1} treats the case of maximal parabolics with Levi isomorphic to $GL_{n_1}\times GL_{n_2}$ with 
$n_1\neq n_2$, and Theorem~\ref{th2} treats the case $n_1=n_2$.  Our method works with only minor differences
for an Eisenstein series or its residue.  However, it is possible that there is no generic residual spectrum for the general linear group induced from cuspidal datum except at the one point we must exclude (the residue at that point is in fact not always generic).  
This question is likely related to the
conjectured generalized Shimura correspondence (see Suzuki \cite{Suz}).

We consider only covers of the general linear group in this article.  In fact for 
covers of other groups we can prove a similar theorem only in special 
cases.  The difficulty in extending the method is due to the property that for the general linear group,
if $\pi$ is not generic then for every 
unipotent orbit in the set ${\mathcal O}(\pi)$ (see below for the definition), the corresponding Fourier coefficient can be written
as an integration over a unipotent subgroup which contains the constant term attached to the
unipotent radical of a maximal parabolic subgroup as an inner integration.
This does not happen in other classical or exceptional groups. Though we consider maximal parabolic subgroups, these methods can be adapted to 
more general parabolic subgroups of the general linear group.  

\section{Notation and Preliminaries}\label{notation}

Let $n_1,n_2\geq1$ and let $P_{n_1,n_2}$ denote the maximal parabolic subgroup of $GL_{n_1+n_2}$ whose Levi part
$M$ is isomorphic to $GL_{n_1}\times GL_{n_2}$ embedded in $GL_{n_1+n_2}$ by $(g_1,g_2)\mapsto \text{diag}(g_1,g_2)$. 
We will assume without loss that $n_1\ge n_2$.
Let $r\geq1$, and let
$F$ be a number field containing a full set of $r$-th roots of unity $\mu_r$. For each $m\geq1$ we let $GL_m^{(r)}(\A)$ denote
an $r$-fold cover of $GL_m(\A)$.  
This is a central extension of $GL_m(\A)$ by $\mu_r$; that is, it consists of ordered pairs $(g,\zeta)$ with $g\in GL_m(\A)$ and 
$\zeta\in \mu_r$, with multiplication 
$(g_1,\zeta_1)(g_2,\zeta_2)=(g_1g_2,\zeta_1\zeta_2\sigma(g_1,g_2))$ where $\sigma:GL_m(\A)^2\to \mu_r$ is a two-cocycle.
See Kazhdan-Patterson \cite{K-P}. There are cocycles giving rise to non-isomorphic extensions 
 but in fact they all agree on the subgroup $GL_m(\A)_0^{(r)}$ consisting of $(g,\zeta)$ with $\det(g)\in (\A^\times)^r F^\times.$
Because of this, we fix one such extension but it does not matter which for the sequel.  If $H$ is any subgroup of $GL_m(\A)$ we
let $\tilde{H}$ denote its inverse image in $GL_m^{(r)}(\A)$.

Let $\tau^{(r)}_i$, $i=1,2$, denote genuine irreducible cuspidal automorphic representations of the groups
$GL_{n_i}^{(r)}(\A)$. Then one can construct an Eisenstein series 
$E_\tau^{(r)}(g,s)$ on $GL_{n_1+n_2}^{(r)}(\A)$.  This depends on a choice of test vectors, but we shall suppress
this from the notation.  This construction is given in detail in Brubaker and Friedberg \cite{B-F}. Alternative constructions
are due to  Suzuki \cite{Suz} and to Takeda \cite{Tak}.  In all cases,
before inducing one must restrict to the subgroup $S:=GL_{n_1}(\A)_0^{(r)}\times_{\mu_r} GL_{n_2}(\A)_0^{(r)}$ 
of $\tilde{M}(\A)$ to take into account that the full
inverse images of the two $GL_{n_i}(\A)$ in $\tilde{M}(\A)$ do not commute.   Since our computations all take place inside
the subgroup of $GL_{n_1+n_2}^{(r)}(\A)$ generated by $S$ and by unipotents, the minor differences in these constructions does not affect the results.  The complex parameter $s$ is normalized so that $E_\tau^{(r)}(g,s)$ has functional equation under $s\mapsto 1-s$.

If $U$ is any unipotent subgroup of $GL_m$, $U(\A)$ splits in the $r$-fold cover
by means of the trivial section $u\mapsto (u,1)$.  We write $U(\A)$ for its isomorphic image in the cover $GL_m^{(r)}(\A)$ under
this section.  If 
$\varphi$ is any automorphic form on $GL^{(r)}_m(\A)$, let $\varphi^U$ denote the constant term of $\varphi$ along $U$
$$\varphi^U(g)=\int_{U(F)\backslash U(\A)} \varphi(ug)\,du.$$
Then the following lemma is standard and follows from the cuspidality of the representations 
$\tau_i^{(r)}$ (see for example M\oe glin-Waldspurger \cite{M-W}, II.1.7).

\begin{lemma}\label{cusp}
Let $E_\tau^{(r)}(g)$ be either a residue of $E_\tau^{(r)}(g,s)$ at a specific point $s_0$, or 
the value of $E_\tau^{(r)}(g,s)$ at  $s=s_0$. If $U$ is the unipotent radical 
of a parabolic subgroup of $GL_{n_1+n_2}$, then the constant term 
$E_\tau^{(r),U}(g)$ is zero for all choices of data unless $U$ is a conjugate of the unipotent radical of $P_{n_1,n_2}$ or of $P_{n_2,n_1}$. 
\end{lemma}

Given an automorphic representation $\pi$ of a reductive group $G$, one may attach a set of integrals over unipotent groups to each
unipotent orbit $\mathcal O$ of $G(\C)$; we call these (generalized) Fourier coefficients.  See Ginzburg \cite{G1}.  The set 
${\mathcal O}(\pi)$ consists of the unipotent orbits $\mathcal O$ with two properties: first, $\pi$ has a non-zero Fourier
coefficient with respect to $\mathcal O$, and second, all Fourier coefficients with respect to unipotent orbits $\mathcal O'$
that are greater than or not comparable with $\mathcal O$ are identically zero.  The definitions extend to metaplectic covers
without change since all unipotent subgroups split in covers. For $G=GL_{n_1+n_2}$, unipotent orbits
correspond to partitions of $n_1+n_2$, by the Jordan decomposition.  

Let ${\mathcal E}_\tau^{(r)}$ be the automorphic representation generated by $E_\tau^{(r)}(g)$
in any one of the cases treated in Lemma~\ref{cusp}.
We have

\begin{proposition}\label{prop1}
The representation ${\mathcal E}_\tau^{(r)}$ is either generic, or ${\mathcal O}({\mathcal E}_\tau^{(r)})=(n_1n_2)$.
\end{proposition}

\begin{proof}
We will assume that ${\mathcal E}_\tau^{(r)}$ is not generic, and prove that ${\mathcal O}({\mathcal E}_\tau^{(r)})=(n_1n_2)$. We fix some notations. 

For $1\le k\le n_1+n_2-1$, let $U_k$ denote the unipotent radical of the standard
parabolic subgroup of $GL_{n_1+n_2}$ whose Levi part is $GL_1^k\times GL_{n_1+n_2-k}$, embedded
in $GL_{n_1+n_2}$ by $(a_1,\ldots,a_k,h)\mapsto \text{diag}(a_1,\ldots,a_k,h)$. 
Let $\psi:F\backslash \A\to\C$ be a fixed nontrivial additive character, and let $\psi_{U_k}$ be the character of $U_k(F)
\backslash U_k(\A)$ given by
$$\psi_{U_k}(u)=\psi(u_{1,2}+u_{2,3}+\cdots +u_{k,k+1}).$$
For $1\le m\le 
n_1+n_2$, let $V_m$ denote the unipotent subgroup of $GL_{n_1+n_2}$ generated by all matrices of
the form $I_{n_1+n_2}+r_{m+1}e_{m,m+1}+\cdots +r_{n_1+n_2}e_{m,n_1+n_2}$, where $e_{i,j}$ is the
matrix whose $(i,j)$ entry is one and whose other entries are zero.  Notice that we can
identify the group $V_{k+1}$ with the quotient $U_{k+1}/U_k$.

For $1\le k\le n_1+n_2-1$ let $I_k$ denote the integral
\begin{equation}\label{int1}
I_k=\int\limits_{U_k(F)\backslash U_k({\A})}E_\tau^{(r)}(ug)\,\psi_{U_k}(u)\,du,
\end{equation}
where $E_\tau^{(r)}(g)$ is a vector in the space of ${\mathcal E}_\tau^{(r)}$.  
It follows from Ginzburg \cite{G1} that this Fourier coefficient is attached to the unipotent orbit  $((k+1)1^{n_1+n_2-k-1})$. Notice also that if $k=n_1+n_2-1$, then the Fourier
coefficient given by \eqref{int1} is the Whittaker coefficient. By our assumption it
is zero for all choices of data.

We will study the vanishing of $I_k$ as we vary over vectors in ${\mathcal E}_\tau^{(r)}$.  To begin,
expand $E_\tau^{(r)}(g)$ along
the unipotent group $U_1$. Since the Weyl group together with any one-parameter unipotent subgroup $x_\alpha(r)$ corresponding
to any root $\alpha$ generates the adelic special
linear group, it follows that a nontrivial automorphic form cannot equal its constant term
along the unipotent radical of a maximal parabolic.  We deduce that for $k=1$ the integral $I_k$
is not zero for some choice of data. 

Let $k_0$ be maximal such that the integral $I_k$ is not zero for some choice of data. We claim first
that $k_0\le n_1-1$. Indeed, suppose $k_0\ge n_1$. Since $k_0\le n_1+n_2-2$, we can expand \eqref{int1} along the group $V_{k_0+1}$. The group $GL_{n_1+n_2-k_0-1}(F)$ acts on the characters appearing in this
expansion with two orbits. First, the nontrivial orbit will contribute a sum such that each of
its summands is the integral $I_{k_0+1}$ (at varying $g$). By the maximality of $k_0$ this sum is zero. The second
contribution to the expansion is from the constant term. However, since $k_0\ge n_1$, it
follows from Lemma~\ref{cusp} that this term is also zero for all choices of data. This is a contradiction.
Thus $k_0\le n_1-1$. 

We next show that $k_0=n_1-1$ or $k_0=n_2-1$.  
Indeed, suppose that $n_2\le k_0\le n_1-1$. If $k_0\neq n_1-1$,  then
as in the previous case, we expand along the group $V_{k_0+1}$. We get
a contradiction, either to the maximality of $k_0$ or to Lemma~\ref{cusp}. Similarly, suppose that
$1\le k_0\le n_2-1$. Then arguing as above we deduce that $k_0=n_2-1$. 

Consider the case when $k_0=n_1-1$. Expand the integral \eqref{int1} along the group $V_{n_1}$. By
the maximality of $k_0$, the contribution to the expansion from the nontrivial orbit is zero.
Thus only the constant term contributes, and we deduce that $I_{n_1-1}$ is equal to
\begin{equation}\label{int2}
I_{n_1}^0=\int\limits_{U_{n_1}(F)\backslash U_{n_1}({\A})}E_\tau^{(r)}(ug)\,\psi(u_{1,2}+u_{2,3}+\cdots +u_{n_1-1,n_1})\,du.
\end{equation}
Next expand the integral \eqref{int2} along the group $V_{n_1+1}$. By Lemma~\ref{cusp}, the contribution
from the constant term is zero. From this we deduce that the integral
\begin{equation}\label{int3}
I_{n_1+1}^0=\int\limits_{U_{n_1+1}(F)\backslash U_{n_1+1}({\A})}E_\tau^{(r)}(ug)\,\psi_{U_{n_1+1}}^0(u)\,du\notag
\end{equation}
is not zero for some choice of data, where
$$\psi_{U_{n_1+1}}^0(u)=\psi(u_{1,2}+u_{2,3}+\cdots +u_{n_1-1,n_1}+u_{n_1+1,n_1+2}).$$
Continuing this process, we deduce that the integral
\begin{equation}\label{int4}
I_{n_1+n_2-1}^0=\int\limits_{U_{n_1+n_2-1}(F)\backslash U_{n_1+n_2-1}({\A})}E_\tau^{(r)}(ug)\,\psi_{U_{n_1+n_2-1}}^0(u)\,du
\end{equation}
is not zero for some choice of data, with
$$\psi_{U_{n_1+n_2+1}}^0(u)=\psi(u_{1,2}+u_{2,3}+\cdots +u_{n_1-1,n_1}+u_{n_1+1,n_1+2}+\cdots
+u_{n_1+n_2-1,n_1+n_2}).$$

A similar result may be obtained in the case when $k_0=n_2-1$. In this case expanding along
$V_{n_2}$, only the constant term contributes a nonzero term to the expansion. After this step, in expanding
in larger $V_m$, only the non-constant
terms contribute.  We deduce that the integral
\begin{equation}\label{int5}
I_{n_1+n_2-1}'=\int\limits_{U_{n_1+n_2-1}(F)\backslash U_{n_1+n_2-1}({\A})}E_\tau^{(r)}(ug)\,\psi_{U_{n_1+n_2-1}}'(u)\,du
\end{equation}
is not zero for some choice of data, where
$$\psi_{U_{n_1+n_2+1}}'(u)=\psi(u_{1,2}+u_{2,3}+\cdots +u_{n_2-1,n_2}+u_{n_2+1,n_2+2}+\cdots
+u_{n_1+n_2-1,n_1+n_2}).$$
Since the Fourier coefficients given by the integrals \eqref{int4} and \eqref{int5} are associated
with the unipotent orbit $(n_1 n_2)$, the Proposition follows. In particular in both cases the integrals are

not zero for some choice of data. 
\end{proof}

\section{Main Theorems}

\subsection{The case when $n_1>n_2$}\label{noteq}
Our first main theorem treats the case that the two factors of the Levi have different sizes. 
\begin{theorem}\label{th1}
Suppose that $n_1>n_2$, and that ${\mathcal E}_\tau^{(r)}$ is the automorphic representation generated by the value of 
$E_\tau^{(r)}(g,s)$ at a specific point $s=s_0$ or by the residue of $E_\tau^{(r)}(g,s)$ at a specific point $s=s_0$, and is not identically zero.
Then ${\mathcal E}_\tau^{(r)}$ is generic.  
\end{theorem}

\begin{proof}
Assume that $\Re(s)$ is sufficiently large that the Eisenstein series $E_\tau^{(r)}(g,s)$
is given as a convergent series. In this case, let $I_{n_1-1}(g,s)$ denote the integral 
\eqref{int1} with $k=n_1-1$. We shall compute this integral in two
ways. Assuming that this Eisenstein series is not generic, we shall derive a contradiction.

First, from the proof of Proposition~\ref{prop1} we have $I_{n_1-1}(g,s)=I_{n_1}^0(g,s)$.
Next we unfold the Eisenstein series. This leads us to consider representatives of the space   
of double cosets $P_{n_1,n_2}(F)\backslash GL_{n_1+n_2}(F)/P_{n_1,n_2}(F)$. 
Using the cuspidality of $\tau^{(r)}_i$ we see that only the identity contributes a non-zero term. So we find that
\begin{equation}\label{id1}
I_{n_1-1}(g,s)=I_{n_1}^0(g,s)=f_{W_1^{(r)},\tau^{(r)}_2}(g,s),
\end{equation}
where $W_1^{(r)}$ denotes the Whittaker coefficient of the representation $\tau^{(r)}_1$
and $f$ is the associated function in the induced space.

To compute $I_{n_1-1}(g,s)$ in a second way, let $w=\begin{pmatrix} &I_{n_2}\\ I_{n_1}&\end{pmatrix}$. Using the left-invariance properties of $E_\tau^{(r)}(g,s)$, we obtain
\begin{equation}\label{id2}
I_{n_1-1}(g,s)=\int\limits_{L(F)\backslash L({\A})}\int\limits_{N_{n_1}(F)\backslash 
N_{n_1}({\A})}E_\tau^{(r)}\left ( 
\begin{pmatrix} I_{n_2}&\\ &u_1\end{pmatrix} \begin{pmatrix} I_{n_2}&\\ l&I_{n_1}
\end{pmatrix}wg,s\right )\,\psi_{N_{n_1}}(u_1)\,du_1\,dl.
\end{equation}
Here $L\cong \Mat_{(n_1-1)\times n_2}$ is the subgroup of $\Mat_{n_1\times n_2}$ consisting of matrices with bottom row zero,
the group $N_{n_1}$ is the maximal upper triangular unipotent subgroup of $GL_{n_1}$, and $\psi_{{N}_{n_1}}$ is the Whittaker character of $N_{n_1}$. 

Next we perform a certain root exchange. The notion of root exchange was defined in Ginzburg-Rallis-Soudry
\cite{G-R-S}, Section 7.1. 
Let $Y_0\cong\Mat_{n_2\times (n_1-1)}$ be the subgroup of $\Mat_{n_2\times n_1}$ consisting of all matrices 
with first column zero.  
Then using \cite{G-R-S}, Lemma 7.1, we deduce that \eqref{id2} is equal to
\begin{multline}
\int\limits_{L({\A})}\int\limits_{Y_0(F)\backslash Y_0({\A})}
\int\limits_{N_{n_1}(F)\backslash N_{n_1}({\A})}
E_\tau^{(r)}\left ( \begin{pmatrix} I_{n_2}&y\\ &I_{n_1}\end{pmatrix}
\begin{pmatrix} I_{n_2}&\\ &u_1\end{pmatrix}  \begin{pmatrix} I_{n_2}&\\ l&I_{n_1}
\end{pmatrix}wg,s\right )\\ \times \psi_{N_{n_1}}(u_1)\,dy\,du_1\,dl.\notag
\end{multline}

Expand the above integral along the unipotent subgroup $V_{n_2+1}^0$ which consists of all
matrices of the form $I_{n_1+n_2}+r_1e_{1,n_2+1}+r_2r_{2,n_2+1}+\cdots + r_{n_2}e_{n_2,n_2+1}$.
The group $GL_{n_2}(F)$ acts on the characters appearing
in this expansion. It follows from Proposition \ref{prop1} and the assumption that $E_\tau^{(r)}(g,s)$ is not generic that ${\mathcal O}( E_\tau^{(r)}(\cdot,s))=(n_1 n_2)$. In particular this means that the contribution to the expansion from the nontrivial orbit is zero. 
Hence only the constant term of the expansion contributes.  
We see that
$I_{n_1-1}(g,s)$ is equal to
\begin{multline}\label{id4}
\int\limits_{L({\A})}
\int\limits_{N_{n_1}(F)\backslash 
N_{n_1}({\A})}\int\limits_{Y(F)\backslash Y({\A})}
E_\tau^{(r)}\left ( \begin{pmatrix} I_{n_2}&y\\ &I_{n_1}\end{pmatrix}
\begin{pmatrix} I_{n_2}&\\ &u_1\end{pmatrix} \begin{pmatrix} I_{n_2}&\\ l&I_{n_1}
\end{pmatrix}wg,s\right )\\ \times \psi_{N_{n_1}}(u_1)\,dy\,du_1\,dl.\notag
\end{multline}
Here $Y=\Mat_{n_2\times n_1}$. 

The integration over $Y$ is the constant term of the Eisenstein series along the unipotent radical of parabolic subgroup $P_{n_2,n_1}$. Unfolding the Eisenstein series, 
we obtain 
\begin{equation}\label{id5}
I_{n_1-1}(g,s)=\int\limits_{L({\A})}\int\limits_{Y({\A})}
f_{W_1^{(r)},\tau^{(r)}_2}\left (w^{-1} \begin{pmatrix} I_{n_2}&y\\ &I_{n_1}\end{pmatrix}
\begin{pmatrix} I_{n_2}&\\ l&I_{n_1}
\end{pmatrix}wg,s\right )\,dy\,dl.
\end{equation}

Now let $g=\begin{pmatrix} aI_{n_1}&\\ &I_{n_2}\end{pmatrix}$, where $a$ is an $r$-th power. Compare the expressions \eqref{id1} and \eqref{id5}. Write $I_{n_1-1}(a,s)$ for $I_{n_1-1}(g,s)$ for the above matrix $g$. From \eqref{id1} we obtain
\begin{equation}\label{id6}
I_{n_1-1}(a,s)=f_{W_1^{(r)},\tau^{(r)}_2}(g,s)=|a|^{n_1n_2s}f_{W_1^{(r)},\tau^{(r)}_2}(e,s)=
|a|^{n_1n_2s}I_{n_1-1}(e,s).
\end{equation}
However, substituting the matrix $g$ into integral \eqref{id5} and conjugating it to the left, we obtain 
\begin{equation}\label{id7}
I_{n_1-1}(a,s)=|a|^{n_1n_2s-n_2}I_{n_1-1}(e,s).
\end{equation}
Here we obtain a factor of $|a|^{(n_1-1)n_2}$ from the change of variables in $L$, and a factor of 
$|a|^{-n_1n_2}$ from the change of variables in $Y$. 

Expressions \eqref{id6} and \eqref{id7} produce a contradiction.  We conclude that
the Eisenstein series $E_\tau^{(r)}(g,s)$ is generic when $\Re(s)$ is large. Also, from the
meromorphic continuation of the Eisenstein series, we deduce that $I_{n_1-1}(g,s)$ is a meromorphic
function of $s$. If we consider the case when  ${\mathcal E}_\tau^{(r)}$ is a value of the Eisenstein series at a specific  point $s_0$, it follows from \eqref{id1} that $I_{n_1-1}(g,s_0)$ is not zero for some choice of data, and hence once again expressions \eqref{id6} and \eqref{id7} produce a contradiction. Finally, if ${\mathcal E}_\tau^{(r)}$ is a nonzero residue of the
Eisenstein series at a point $s_0$, then it follows  from the equality \eqref{id1} that 
the residue at $s_0$ of $I_{n_1-1}(g,s)$ is zero for all choices of data. However, this residue is an inner integration
in \eqref{int4} in the case that $E_\tau^{(r)}(g)={\rm Res}_{s=s_0} E_\tau^{(r)}(g,s)$.  This gives a contradiction to Proposition \ref{prop1}.

This completes the proof of Theorem~\ref{th1}.\end{proof}

\subsection{The case when $n_1=n_2$}\label{eq}
In this section we treat the case $n_1=n_2$.  This case is fundamentally different, for there are situations in which the 
representation ${\mathcal E}_\tau^{(r)}$ is not generic. 
Indeed, let $n_1=n_2=2$. Take a cuspidal theta representation $\Theta_\chi^{(2)}$ defined
on the group $GL_2^{(2)}({\A})$ where $\chi$ is a character of $F^\times\backslash {\A}^\times$ which is not the square of a character. These representations were constructed in Gelbart and Piatetski-Shapiro  \cite{G-PS}. Let 
$\tau_i^{(2)}=\Theta_\chi^{(2)}$ for $i=1,2$. Then it is not hard to check that the corresponding Eisenstein series has a simple pole at $s=3/4$, and the residue representation is not generic. Indeed, the unramified constituent of the residue at a finite place $\nu$ is the theta representation of the group $GL_4^{(2)}$ with $\chi_\nu$ as its central character. It is well known that this representation is not generic. Thus the residue is not generic.

We shall show:

\begin{theorem}\label{th2}
Suppose that $n_1=n_2=n$, and that ${\mathcal E}_\tau^{(r)}$ is the automorphic representation generated by the value of 
$E_\tau^{(r)}(g,s)$ at a specific point $s=s_0$,
and is not identically zero.  Then ${\mathcal E}_\tau^{(r)}$ is generic.  Suppose instead that ${\mathcal E}_\tau^{(r)}$ is the automorphic representation generated by the residue of $E_\tau^{(r)}(g,s)$ at a specific point $s=s_0$, and is not identically zero.
Then it is generic provided $s_0\neq \frac{n\pm1}{2n}$.
\end{theorem}

\begin{proof}
As in the previous Section we shall assume that the representation ${\mathcal E}_\tau^{(r)}$ is not
generic and determine the cases in which we can derive a contradiction.   Though the argument is
similar to the one above, there is an important technical difference.  
When $n_1=n_2$, there is only one unipotent radical of a parabolic (up to conjugation) that may have a nonzero constant term,
rather than two distinct ones.  See Lemma~\ref{cusp}.

Let ${\mathcal E}_\tau^{(r)}$ be as above, and write $n=n_1=n_2$. Assume that $\Re(s)$ is large. It follows
from Proposition~\ref{prop1} that ${\mathcal O}({\mathcal E}_\tau^{(r)})=(n^2)$. The corresponding Fourier coefficient is described in a more general way  in Friedberg-Ginzburg \cite{F-G1}, equations (6) and (7). We recall its definition.

Consider the subgroup $V_{2,n}$ of $GL_{2n}$ consisting of all matrices of the form
\begin{equation*}
\begin{pmatrix} I_{2}&X_{1,2}&*&*&\cdots&*\\ &I_{2}&X_{2,3}&*&\cdots&*\\ &&I_{2}&X_{3,4}&\cdots&*\\
&&&I_{2}&\cdots&*\\ &&&&\ddots&*\\ &&&&&I_{2}\end{pmatrix}
\end{equation*}
with $I_2$ appearing $n$ times and each $X_{i,j}$ a matrix of size $2$. Define a character
$\psi_{V_{2,m}}$ on $V_{2,n}(F)\backslash V_{2,n}(\A)$ by 
$$\psi_{V_{2,n}}(v)=\psi(\text{tr} (X_{1,2}+X_{2,3}+\cdots +X_{n-1,n})).$$ 
The corresponding Fourier coefficient is given by
\begin{equation}\label{four1}
\int\limits_{V_{2,n}(F)\backslash V_{2,n}({\A})}E_\tau^{(r)}(vg,s)\,\psi_{V_{2,n}}(v)\,dv.
\end{equation}

Consider the diagonal embedding of $GL_2$ into $GL_{2n}$, $h\mapsto 
\iota(h)=\text{diag}(h,h,\ldots,h)$.  If $h\in GL_2(F)$, then it is immediate that the integral
\eqref{four1} is left-invariant under $\iota(h)$. However, if $r$ does not divide $n$, then $\iota$
does not extend to an embedding of $GL_2(\A)$ into the $r$-fold cover. 
This means that restricting to $g=\iota(h)$, the integral \eqref{four1} defines a $GL_2(F)$-invariant function 
on a non-trivial cover of $GL_2(\A)$, and which can thus not be the trivial
function. Arguing as in Friedberg-Ginzburg \cite{F-G1}, Proposition 3 ((9) and following), we deduce that
the representation ${\mathcal E}_\tau^{(r)}$ has a nontrivial Fourier coefficient corresponding
to the unipotent orbit $((n+1)(n-1))$. This a contradiction. 

Hence we may assume that $r$ divides $n$, and that the integral \eqref{four1} is left-invariant 
under all matrices $\iota(h)$ with $h\in GL_2({\A})$. Let $h=\begin{pmatrix} a&\\ &a^{-1}\end{pmatrix}$, and
let $I(a,s)$ denote the integral \eqref{four1} with $g=\iota(h)$. It follows from the 
left invariance properties of the above Fourier coefficient that $I(a,s)=I(e,s)$. 

Next we compute the integral \eqref{four1} in a different way. 
Let $w$ denote the Weyl element of $GL_{2n}$ with $w_{i,2i-1}=w_{n+i,2i}=1$ for all $1\le i\le n$ and all other entries zero. 
Conjugating by this element from left to right we obtain
\begin{multline}\label{four2}
I(a,s)=\int\limits_{L(F)\backslash L({\A})}
\int\limits_{N_{n}(F)\backslash N_{n}({\A})}\int\limits_{N_{n}(F)\backslash N_{n}({\A})} \int\limits_{Y(F)\backslash Y({\A})}\notag
E_\tau^{(r)}\left (\begin{pmatrix} u_1&y\\ &u_2\end{pmatrix}
\begin{pmatrix} I_n&\\ l&I_n\end{pmatrix}wg,s\right )\\ \times \psi_{N_{n}}(u_1)\,\psi_{N_{n}}(u_2)\,
dy\,du_1\,du_2\,dl.
\end{multline}
Here we change the notation from prior Sections:
the groups $L$ and $Y$ are each the group
consisting of all matrices in $x\in \Mat_{n\times n}$ such that $x_{i,j}=0$ for all $j\le i$. 

We now perform a series of root exchanges together with certain Fourier expansions. Using the fact
that ${\mathcal O}({\mathcal E}_\tau^{(r)})=(n^2)$, we obtain 
\begin{multline}
I(a,s)=\int\limits_{ L({\A})}
\int\limits_{N_{n}(F)\backslash N_{n}({\A})}\int\limits_{N_{n}(F)\backslash N_{n}({\A})} 
\int\limits_{Y'(F)\backslash Y'({\A})}\notag
E_\tau^{(r)}\left (\begin{pmatrix} u_1&y'\\ &u_2\end{pmatrix}
\begin{pmatrix} I_n&\\ l&I_n\end{pmatrix}wg,s\right )\\ \times \psi_{N_{n}}(u_1)\,\psi_{N_{n}}(u_2)\,
dy'\,du_1\,du_2\,dl.
\end{multline}
Here $Y'=\Mat_{n\times n}$. Notice that the integral over $Y'$ is the constant term along the unipotent radical
of the parabolic subgroup $P_{n,n}$. Hence, since $\Re(s)$ is large,  unfolding the Eisenstein series, we deduce that $I(a,s)$ is equal to  
\begin{equation}\label{four4}
\int\limits_{ L({\A})}
f_{W_1,W_2}\left (\begin{pmatrix} I_n&\\ l&I_n\end{pmatrix}wg,s\right)dl+
\int\limits_{ L({\A})}\int\limits_{Y'({\A})}
f_{W_1,W_2}\left (w_1\begin{pmatrix} I_n&y'\\ &I_n\end{pmatrix}\begin{pmatrix} I_n&\\ l&I_n\end{pmatrix}wg,s\right)dy'\,dl.\notag
\end{equation}
Here $w_1=\begin{pmatrix} &I_n\\ I_n&\end{pmatrix}$, and $W_i$ denotes the Whittaker coefficient of the representation $\tau_i^{(r)}$. 

Suppose that $a$ is an $r$-th power. Conjugating the matrix $g$ to the left we obtain 
\begin{equation}\label{four5}
I(a,s)=|a|^{2n^2s-n^2+n}I_1(s)+|a|^{-2n^2s+n^2+n}I_2(s).
\end{equation}
Here
\begin{equation}\label{first}
I_1(s)=\int\limits_{ L({\A})}
f_{W_1,W_2}\left (\begin{pmatrix} I_n&\\ l&I_n\end{pmatrix}w,s\right)dl\notag
\end{equation}
and 
\begin{equation}\label{second}
I_2(s)=\int\limits_{ L({\A})}\int\limits_{Y'({\A})}
f_{W_1,W_2}\left (w_1\begin{pmatrix} I_n&y'\\ &I_n\end{pmatrix}\begin{pmatrix} I_n&\\ l&I_n\end{pmatrix}w,s\right)dy'\,dl.\notag
\end{equation}
Since $I(a,s)=I(e,s)$ (and so is independent of $a$ and also non-zero), and since the numbers $2n^2s-n^2+n$ and $-2n^2s+n^2+n$ cannot vanish at the
same value of $s$,  \eqref{four5} will produce a contradiction unless either $ I_1(s)$ is zero
for all choices of data and $-2n^2s+n^2+n=0$, or $I_2(s)=0$ for all choices of data and $2n^2s-n^2+n=0$. From this we
deduce that for $\Re(s)$ large, the representation ${\mathcal E}_\tau^{(r)}$ must be generic. 
This is also true if $s$ is not equal to $\frac{n-1}{2n}$ or to $\frac{n+1}{2n}$.

Suppose that $s=\frac{n+1}{2n}$. Arguing similarly to Friedberg-Ginzburg \cite{F-G1}, it is not hard to check that
the integral $I_1(s)$ is zero for all choices
of data if and only if $f_{W_1,W_2}(e,s)$ is zero for all choices of data. Hence, unless 
${\mathcal E}_\tau^{(r)}$ is the residue at $s=\frac{n+1}{2n}$, then it must be generic. 
By the functional equation of the Eisenstein series  (M\oe glin-Waldspurger \cite{M-W}), the same holds at $s=\frac{n-1}{2n}$.

This completes the proof of the Theorem.
\end{proof}

In a similar way one may show that an Eisenstein series associated to a non-maximal parabolic formed from
irreducible cuspidal representations is once again
generic as long as the complex tuple $\mathbf s$ is in general position.

\end{document}